\documentclass[12pt,reqno]{amsart}
\usepackage{hyperref}
\usepackage{amsmath,amssymb}
\usepackage{caption}
\usepackage{graphicx}
\usepackage{float}
\usepackage{pgf,tikz}
\usepackage{mathrsfs}

\newcommand \reg{\operatorname{reg}}

\newtheorem{theorem}{Theorem}[section]
\newtheorem{definition}[theorem]{Definition}
\newtheorem{lemma}[theorem]{Lemma}
\newtheorem{proposition}[theorem]{Proposition}
\newtheorem{example}[theorem]{Example}

\newtheorem{remark}[theorem]{Remark}
\newtheorem{corollary}[theorem]{Corollary}

\begin{document}
\title{Regularity of binomial edge ideals of Cohen-Macaulay bipartite graphs}
\author{A V Jayanthan}
\email{jayanav@iitm.ac.in}

\author{Arvind Kumar}
\email{arvkumar11@gmail.com}
\address{Department of Mathematics, Indian Institute of Technology
Madras, Chennai, INDIA - 60036}

\maketitle

\begin{abstract}
Let $G$ be a finite simple graph on $n$ vertices and $J_G$ denote the
corresponding binomial edge ideal in $S = K[x_1, \ldots, x_n, y_1,
\ldots, y_n].$ We compute the Castelnuovo-Mumford regularity of $S/J_G$
when $J_G$ is the binomial edge ideal of a Cohen-Macaulay bipartite
graph. We achieve this by computing the regularity of certain
bipartite subgraphs and some intermediate graphs, called $k$-fan
graphs. In this process, we also obtain a class of graphs which
satisfy the regularity conjecture of Saeedi Madani and Kiani.
\end{abstract}

\section{Introduction}
Let $G$ be a finite simple graph on $[n]$. Herzog et al. in \cite{HH1}
and independently Ohtani in \cite{oh}, introduced the notion of
binomial edge ideal corresponding to a finite simple graph.  Let
$S=K[x_1, \ldots, x_n,y_1, \ldots, y_n]$, where $K$ is a field.  The
binomial edge ideal of the graph $G$ is  $J_G =(x_i y_j - x_j y_i :
\{i,j\} \in E(G), \; i <j)$. Researchers have been trying to establish
connections between combinatorial invariants associated to $G$ and
algebraic invariants associated to $J_G$. In particular, connections
have been established between homological invariants such as depth,
codimension, Betti numbers and Castelnuovo-Mumford regularity of $J_G$
with certain combinatorial invariants associated to $G$, see for
example \cite{her1,HH1,JNR,KM3,MM,Rauf,KM1,KM2}.  In \cite[Theorem
1.1]{MM}, Matsuda and Murai proved that for any graph $G$ on vertex
set $[n]$, $l \leq \reg(S/J_G) \leq n-1$, where $l$ is the length of
longest induced path in $G$. They conjectured that $\reg(S/J_G) = n-1$
if and only if $G$ is the path graph. This conjecture was settled in
affirmative by Kiani and Saeedi Madani in \cite{KM3}. For a graph $G$,
let $c(G)$ denote the number of maximal cliques of $G$. If $G$ is a
closed graph, i.e., if $J_G$ has a quadratic Gr\"obner basis, then
Saeedi Madani and Kiani proved that $\reg(S/J_G) \leq c(G)$,
\cite{KM1}. They conjectured that $\reg(S/J_G) \leq c(G)$ for any
finite simple graph $G$, \cite{KM2}.  In \cite{KM5}, Saeedi Madani and
Kiani proved the conjecture  for generalized block graphs.

Another homological invariant associated with an ideal $I$ is the
depth of $S/I$. While not much is known about the depth of binomial
edge ideals, there are some results on the structure of certain
classes of graphs whose binomial edge is Cohen-Macaulay, which
corresponds to the highest possible depth.  Bolognini et al. studied
the structure of bipartite graphs and characterized the
Cohen-Macaulayness of the binomial edge ideals of bipartite graphs in
\cite{dav}.  They introduced  a family of bipartite graphs, denoted by
$F_m$ and a family of non-bipartite graphs, called $k$-fan graphs,
denoted by $F_k^W(K_n)$, whose binomial edge ideals are Cohen-Macaulay
(see Sections 2 and 3 for the definition).

There are very few classes of graphs for which the regularity of their
binomial edge ideals is known. The upper and lower bounds known are,
in general, far from being sharp for most of the classes of graphs.
In this article, we compute the regularity of the binomial edge ideals
of Cohen-Macaulay bipartite graphs. First, we show that
the $k$-fan graphs $F^W_k(K_n)$ satisfy the upper bound conjectured
by Saeedi Madani and Kiani. It may be noted that $F^W_k(K_n)$ is not
necessarily a generalized block graph or a closed graph. We
also obtain a subclass which attains the upper bound, (Theorem
\ref{3.3}). We then compute the regularity of $k$-pure fan graphs,
(Theorem \ref{3.4}). In \cite{dav}, it
was proved that if $G$ is a connected bipartite graph, then $J_G$ is
Cohen-Macaulay if and only if $G = G_1 * \cdots * G_s$, where $G_i = F_m$ or
$G_i = F_{m_1} \circ \cdots \circ F_{m_t}$ for some $m \geq 1$ and
$m_j \geq 3$, see Section 2 for the
definition of the operations $\circ$ and $*$. By
\cite[Theorem 3.1]{JNR}, it is known that if $G = G_1 * G_2$, then
$\reg(S/J_G) = \reg(S/J_{G_1}) + \reg(S/J_{G_2})$. Therefore, to
compute the regularity of Cohen-Macaulay bipartite graphs, we need to
understand the regularity behavior under the operation $\circ$.
We first show that
$\reg(S/J_{F_m}) = 3$ if $m \geq 2$, (Proposition \ref{4.1}). 
We then compute the regularity of the intermediate graphs such as
$F_{m_1} \circ \cdots \circ F_{m_t}\circ H$, where $H$ is either $F_n$
 or  a fan graph $F_k^W(K_n)$ for some $n \geq 3$, (Theorem \ref{4.6}).
Using these information, we obtain a precise expression for the
regularity of binomial edge ideals of Cohen-Macaulay bipartite graphs,
(Theorem \ref{cm-bipartite}).

\section{Preliminaries}
In this section, we  recall  some notation and fundamental results on
graphs and the corresponding binomial edge ideals which are used
throughout this paper.

Let $G$  be a  finite simple graph with vertex set $V(G)$ and edge set
$E(G)$. A graph $G$ is said to be \textit{bipartite} if 
there is a bipartition of $V(G)=V_1 \sqcup V_2$ such that for each
$i=1,2$, no two of the vertices of $V_i$ are adjacent. 
For a subset $A \subseteq V(G)$, $G[A]$ denotes the induced
subgraph of $G$ on the vertex set $A$, that is, for $i, j \in A$, $\{i,j\} \in E(G[A])$ if and only if $ \{i,j\} \in E(G)$. 
For a vertex $v$, $G \setminus v$ denotes the  induced subgraph of $G$
on the vertex set $V(G) \setminus \{v\}$. A vertex $v \in V(G)$ is
said to be a \textit{cut vertex} if $G \setminus v$ has strictly more
connected components than $G$.
A subset $U$ of $V(G)$ is said to be a 
\textit{clique} if $G[U]$ is a complete graph. 
A vertex $v$ is said to be a \textit{free vertex} if it belongs to
exactly one maximal clique. For a vertex $v$, $N_G(v) = \{u \in V(G) ~
: ~ \{u,v\} \in E(G)\}$ (neighborhood of $v$),  $N_G[v] = N_G(v)
\cup \{v\}$ and $\deg_G(v) = |N_G(v)|$.  A vertex $v$ is said to be
pendant vertex, if $\deg_G(v) =1$.  For a vertex $v$, $G_v$ is the
graph on vertex set $V(G)$ and edge set $E(G_v) =E(G) \cup \{
\{u,w\}: u,w \in N_G(v)\}$. 
We say that a graph $G$ is Cohen-Macaulay if $S/J_G$ is Cohen-Macaulay.
For a graph $G$, by \textit{regularity of} $G$, we mean the regularity
of the binomial edge ideal of $G$.
 
For every $m\geq 1$, $F_m$ denotes the graph on the vertex set $[2m]$ and edge set $E(F_m) =\{ \{2i,2j-1\} : i=1,\ldots,m,j=i,\ldots,m\}$. 
It was shown that the graphs $F_m$'s form basic building blocks of
Cohen-Macaulay bipartite graphs, see \cite{dav} for details.
Here we recall from \cite{dav} the two operations, denoted by $*$ and
$\circ$, which are important in the study of Cohen-Macaulay bipartite
graphs.

\textbf{Operation $*$} : For $i=1,2$, let $G_i$ be a graph with at
least one free vertex $f_i$.  We denote by $G =(G_1,f_1) * (G_2,f_2)$
the graph obtained by identifying the vertices $f_1$ and $f_2$.

\textbf{Operation $\circ$} : For $i=1,2$, let $G_i$ be a graph with at least one pendant vertex $f_i$ and $v_i$ be its neighbor with 
$\deg_{G_i}(v_i) \geq 2$. Then we define $G=(G_1,f_1) \circ (G_2, f_2)$
to be the graph obtained from $G_1$ and $G_2$ by removing the pendant
vertices $f_1,f_2$ and identifying the vertices $v_1$ and $ v_2$.
 
In the above notation, we may suppress $f_1$ and $f_2$ whenever it is
not necessary to emphasize them. Below, we illustrate the definition
of $*$ and $\circ$ with an example.

Let $G$ and $H$ be the graphs as given below:
\vskip 2mm \noindent
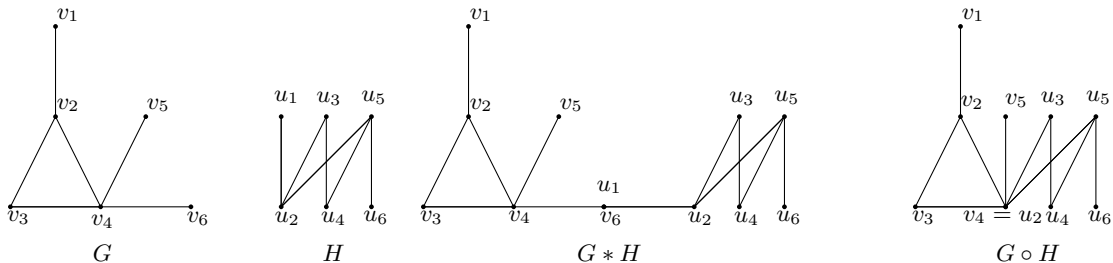
\begin{minipage}{\linewidth}
  \begin{minipage}{0.3\linewidth}
	\captionsetup[figure]{labelformat=empty}
\begin{figure}[H]
\begin{tikzpicture}[scale=.6]
\draw (4,2)-- (4,0);
\draw (4,0)-- (3,-2);
\draw (4,0)-- (5,-2);
\draw (3,-2)-- (5,-2);
\draw (5,-2)-- (6,0);
\draw (5,-2)-- (7,-2);
\draw (3,-2)-- (5,-2);
\draw (9,0)-- (9,-2);
\draw (9,-2)-- (10,0);
\draw (10,0)-- (10,-2);
\draw (10,-2)-- (11,0);
\draw (11,0)-- (11,-2);
\draw (9,-2)-- (11,0);
\draw (9,0)-- (9,-2);
\draw (9,-2)-- (11,0);
\begin{scriptsize}
\fill (4,2) circle (1.5pt);
\draw (4.28,2.26) node {$v_1$};
\fill (4,0) circle (1.5pt);
\draw (4.28,0.26) node {$v_2$};
\fill (3,-2) circle (1.5pt);
\draw (3.18,-2.24) node {$v_3$};
\fill (5,-2) circle (1.5pt);
\draw (5.04,-2.32) node {$v_4$};
\fill (6,0) circle (1.5pt);
\draw (6.28,0.26) node {$v_5$};
\fill (7,-2) circle (1.5pt);
\draw (7.16,-2.26) node {$v_6$};
\fill (9,-2) circle (1.5pt);
\draw (9.14,-2.26) node {$u_2$};
\fill (10,-2) circle (1.5pt);
\draw (10.16,-2.26) node {$u_4$};
\fill (11,-2) circle (1.5pt);
\draw (11.1,-2.24) node {$u_6$};
\fill (9,0) circle (1.5pt);
\draw (9.12,0.44) node {$u_1$};
\fill (10,0) circle (1.5pt);
\draw (10.06,0.4) node {$u_3$};
\fill (11,0) circle (1.5pt);
\draw (11.06,0.42) node {$u_5$};
\draw (5.04,-3.02) node {$G$};
\draw (10.16,-3.02) node {$H$};
\end{scriptsize}
\end{tikzpicture}
\end{figure}
\end{minipage}
\begin{minipage}{0.40\linewidth}
\captionsetup[figure]{labelformat=empty}
\begin{figure}[H]
\begin{tikzpicture}[scale=.6]
\draw (4,2)-- (4,0);
\draw (4,0)-- (3,-2);
\draw (4,0)-- (5,-2);
\draw (3,-2)-- (5,-2);
\draw (5,-2)-- (6,0);
\draw (5,-2)-- (7,-2);
\draw (3,-2)-- (5,-2);
\draw (7,-2)-- (9,-2);
\draw (9,-2)-- (10,0);
\draw (10,0)-- (10,-2);
\draw (10,-2)-- (11,0);
\draw (11,0)-- (11,-2);
\draw (9,-2)-- (11,0);
\draw (7,-2)-- (9,-2);
\draw (9,-2)-- (11,0);
\begin{scriptsize}
\fill  (4,2) circle (1.5pt);
\draw (4.28,2.26) node {$v_1$};
\fill  (4,0) circle (1.5pt);
\draw (4.28,0.26) node {$v_2$};
\fill  (3,-2) circle (1.5pt);
\draw (3.18,-2.24) node {$v_3$};
\fill  (5,-2) circle (1.5pt);
\draw (5.14,-2.22) node {$v_4$};
\fill  (6,0) circle (1.5pt);
\draw (6.28,0.26) node {$v_5$};
\fill  (7,-2) circle (1.5pt);
\draw (7.16,-2.26) node {$v_6$};
\fill  (9,-2) circle (1.5pt);
\draw (9.14,-2.26) node {$u_2$};
\fill  (10,-2) circle (1.5pt);
\draw (10.16,-2.26) node {$u_4$};
\fill  (11,-2) circle (1.5pt);
\draw (11.1,-2.24) node {$u_6$};
\fill  (7,-2) circle (1.5pt);
\draw (7.12,-1.56) node {$u_1$};
\fill  (10,0) circle (1.5pt);
\draw (10.06,0.4) node {$u_3$};
\fill  (11,0) circle (1.5pt);
\draw (11.06,0.42) node {$u_5$};
\draw (7.12,-3.02) node {$G*H$};
\end{scriptsize}
\end{tikzpicture}
\end{figure}
\end{minipage}
\begin{minipage}{.25\linewidth}
\captionsetup[figure]{labelformat=empty}
\begin{figure}[H]
\begin{tikzpicture}[scale=.6]
\draw (4,2)-- (4,0);
\draw (4,0)-- (3,-2);
\draw (4,0)-- (5,-2);
\draw (3,-2)-- (5,-2);
\draw (5,-2)-- (5,0);
\draw (3,-2)-- (5,-2);
\draw (5,-2)-- (6,0);
\draw (6,0)-- (6,-2);
\draw (6,-2)-- (7,0);
\draw (7,0)-- (7,-2);
\draw (5,-2)-- (7,0);
\draw (5,-2)-- (7,0);
\begin{scriptsize}
\fill  (4,2) circle (1.5pt);
\draw (4.28,2.26) node {$v_1$};
\fill  (4,0) circle (1.5pt);
\draw (4.26,0.32) node {$v_2$};
\fill  (3,-2) circle (1.5pt);
\draw (3.18,-2.24) node {$v_3$};
\fill  (5,-2) circle (1.5pt);
\draw (4.78,-2.22) node {};
\fill  (5,0) circle (1.5pt);
\draw (5.24,0.36) node {$v_5$};
\fill  (5,-2) circle (1.5pt);
\draw (4.95,-2.22) node {$v_4=u_2$};
\fill  (6,-2) circle (1.5pt);
\draw (6.16,-2.26) node {$u_4$};
\fill  (7,-2) circle (1.5pt);
\draw (7.1,-2.24) node {$u_6$};
\fill  (6,0) circle (1.5pt);
\draw (6.06,0.4) node {$u_3$};
\fill  (7,0) circle (1.5pt);
\draw (7.06,0.42) node {$u_5$};
\draw (5.50,-3.02) node {$G\circ H$};
\end{scriptsize}
\end{tikzpicture}
\end{figure}
\end{minipage}
\end{minipage}

The graph $G*H$ given above is obtained by identifying the vertices
$v_6$ of $G$ and $u_1$ of $H$. By deleting the vertices $v_6$ of $G$ and
$u_1$ of $H$ and identifying the vertices $v_4$ and $u_2$, we obtain
$G\circ H$ as given above.

For a subset $T$ of $[n]$, let $\bar{T} = [n]\setminus T$ and $c_G(T)$
denote the number of connected components of $G[\bar{T}]$. Let $G_1,\cdots,G_{c_{G(T)}}$ be connected 
components of $G[\bar{T}]$. For each $i$, let $\tilde{G_i}$ denote the complete graph on $V(G_i)$ and
$$P_T(G) = (\underset{i\in T} \cup \{x_i,y_i\}, J_{\tilde{G_1}},\cdots, J_{\tilde{G}_{c_G(T)}}).$$ 
It was shown by Herzog et al. that $J_G =  \underset{T \subseteq [n]}\cap P_T(G)$, \cite{HH1}.
For each $i \in T$, if $i$ is a cut vertex of the graph $G[\bar{T} \cup \{i\}]$,
then we say that $T$ has the cut point property. Set $\mathscr{C}(G) =\{\phi \}
\cup \{ T: T \; \text{has the cut point property} \}$.
Throughout this paper, we  use a short exact sequence which allows us
to use induction.
\begin{remark}\label{2.5}
Let $G$ be a finite simple graph and $v$ be a vertex which is not a
free vertex in $G$.   In \cite[Lemma 4.8]{oh}, it was shown that
$J_G = Q_1 \cap Q_2$, where  $Q_1 = J_{G_v}$, $Q_2 = (x_v,y_v) +
J_{G\setminus v}$ and $Q_1 +Q_2 = (x_v,y_v) + J_{G_v \setminus
v}$. This gives rise to the following short exact sequence:
 \begin{equation}\label{2.6}
  0  \longrightarrow  \dfrac{S}{J_{G} } \longrightarrow
  \dfrac{S}{Q_1} \oplus \dfrac{S}{Q_2} \longrightarrow \dfrac{S}{Q_1
  +Q_2 } \longrightarrow 0.
\end{equation}
\end{remark}

The following basic property of regularity is used repeatedly in this
article.
 
\begin{lemma}\label{2.4}
 Let $S$ be a standard graded ring  and $M,N$ and $P$ be finitely generated graded $S$-modules. 
 If $ 0 \rightarrow M \xrightarrow{f}  N \xrightarrow{g} P \rightarrow 0$ is a 
 short exact sequence with $f,g$  
 graded homomorphisms of degree zero, then 
 \begin{enumerate}
  \item $reg(M) \leq max\{reg(N),reg(P)+1\}$.
  \item $reg(M) = reg(N)$, if $reg(N) >reg(P)$. 
 \end{enumerate}
\end{lemma}
 
\section{Regularity of $F_k^W(K_n)$}
In \cite{dav}, Bolognini et al. introduced a family of chordal graphs namely the 
fan of a complete graph $K_n$.

\begin{definition}
Let $K_n$ be the complete graph on the vertex set $[n]$ and $W =
\{v_1, \ldots , v_r\} \subseteq [n]$. Then $F^W(K_n)$ is the graph
obtained from $K_n$ by the following operation:  for every $i = 1,
\ldots ,r$,  attach a complete graph $K_{a_i}$ to $K_n$ in such a way
that $V(K_n) \cap  V(K_{a_i}) = \{v_1, \ldots, v_i\}$, for some $a_i >
i$. We say that the graph $F^W(K_n)$ is obtained by adding a fan to
$K_n$ on the set $W$ and $\{ K_{a_1},\ldots, K_{a_r}\} $ is the branch of
that fan on  $W$.

Let $K_n$ be the complete graph on $[n]$ and $W_1 \sqcup \cdots \sqcup
W_k$ be a partition of a subset $W \subseteq [n]$.  Let $F_k^W(K_n)$
be the graph
obtained from $K_n$ by adding a fan on each set $W_i$.  For each $i\in
\{1,\ldots , k \}$, set $W_i =\{v_{i,1},\ldots , v_{i,r_i}\}$ and
$\{K_{a_{i,1}},\ldots , K_{a_{i,r_i}}\}$ be the branch of the fan on
$W_i$. The graph $F_k^W(K_n)$ is called a $k$-fan of $K_n$ on the set
$W$.

A branch $\{K_{a_{i,1}},\ldots,K_{a_{i,r_i}}\} $ of the fan on $W_i$
is said to be a pure branch if for each $j =1,\ldots,r_i$, $a_{i,j} =
j+1$.  A fan on set $W_i $ is said to be pure fan, if its branch is
pure. If for each $i \in \{1,\ldots,k\}$, fan on set $W_i$ is pure,
then $F_k^W(K_n)$ is said to be a $k$-pure fan graph of $K_n$ on $W$.
\end{definition}

 \begin{example}
Let $G_1$ and $G_2$ be the graphs as shown in the figure below.
  
\begin{minipage}{\linewidth}
\begin{minipage}{0.25\linewidth}
\definecolor{qqqqff}{rgb}{0,0,1}
\captionsetup[figure]{labelformat=empty}
\begin{figure}[H]
\begin{tikzpicture}[scale=.6]
\draw (-1,4)-- (-0.94,2.68);
\draw (-0.94,2.68)-- (0.02,1.6);
\draw (0.02,1.6)-- (1.04,2.74);
\draw (1.04,2.74)-- (1,4);
\draw (1,4)-- (0,5);
\draw (0,5)-- (-1,4);
\draw (-1,4)-- (-2,5);
\draw (-1,4)-- (-1.16,5.66);
\draw (-1.16,5.66)-- (0,5);
\draw (0,5)-- (0,6);
\draw (0,6)-- (-1,4);
\draw (-1,4)-- (1,4);
\draw (1,4)-- (0,6);
\draw (0,5)-- (-0.94,2.68);
\draw (-0.94,2.68)-- (1.04,2.74);
\draw (1.04,2.74)-- (0,5);
\draw (-1,4)-- (0.02,1.6);
\draw (0.02,1.6)-- (1,4);
\draw (1,4)-- (-0.94,2.68);
\draw (-1,4)-- (1.04,2.74);
\draw (0,5)-- (0.02,1.6);
\draw (1.04,2.74)-- (2.14,2.92);
\draw (1.04,2.74)-- (1.34,1.38);
\draw (1.34,1.38)-- (0.02,1.6);
\begin{scriptsize}
\fill (0,5) circle (1.5pt);
\draw(0.14,5.26) node {$2$};
\fill (-1,4) circle (1.5pt);
\draw(-1.26,3.96) node {$1$};
\fill (1,4) circle (1.5pt);
\draw(1.16,4.26) node {$3$};
\fill (1.04,2.74) circle (1.5pt);
\draw(1.18,3) node {$4$};
\fill (-0.94,2.68) circle (1.5pt);
\draw(-1.22,2.62) node {$6$};
\fill (0.02,1.6) circle (1.5pt);
\draw(-0.04,1.38) node {$5$};
\fill (-2,5) circle (1.5pt);
\fill (-1.16,5.66) circle (1.5pt);
\fill (0,6) circle (1.5pt);
\fill (2.14,2.92) circle (1.5pt);
\fill (1.34,1.38) circle (1.5pt);
\end{scriptsize}
\end{tikzpicture}
\caption{$G_1$}
\end{figure}
\end{minipage}
\begin{minipage}{0.25\linewidth}
\captionsetup[figure]{labelformat=empty}
\begin{figure}[H]
\definecolor{qqqqff}{rgb}{0,0,1}
\begin{tikzpicture}[scale=.6]
\draw (-1,4)-- (-0.94,2.68);
\draw (-0.94,2.68)-- (0.02,1.6);
\draw (0.02,1.6)-- (1.04,2.74);
\draw (1.04,2.74)-- (1,4);
\draw (1,4)-- (0,5);
\draw (0,5)-- (-1,4);
\draw (-1,4)-- (-2,5);
\draw (-1,4)-- (-1.16,5.66);
\draw (-1.16,5.66)-- (0,5);
\draw (0,5)-- (0,6);
\draw (0,6)-- (-1,4);
\draw (-1,4)-- (1,4);
\draw (1,4)-- (0,6);
\draw (0,5)-- (-0.94,2.68);
\draw (-0.94,2.68)-- (1.04,2.74);
\draw (1.04,2.74)-- (0,5);
\draw (-1,4)-- (0.02,1.6);
\draw (0.02,1.6)-- (1,4);
\draw (1,4)-- (-0.94,2.68);
\draw (-1,4)-- (1.04,2.74);
\draw (0,5)-- (0.02,1.6);
\draw (1.04,2.74)-- (2.14,2.92);
\draw (1.04,2.74)-- (1.34,1.38);
\draw (1.34,1.38)-- (0.02,1.6);
\draw (-2,5)-- (-2.28,3.96);
\draw (-2.28,3.96)-- (-1,4);
\draw (-1.52,5.08)-- (-1,4);
\draw (-1.52,5.08)-- (0,5);
\draw (-1.52,5.08)-- (-1.16,5.66);
\draw (1.04,2.74)-- (0.88,1.88);
\draw (0.88,1.88)-- (1.34,1.38);
\draw (0.88,1.88)-- (0.02,1.6);
\begin{scriptsize}
\fill  (0,5) circle (1.5pt);
\draw (0.32,5.22) node {$2$};
\fill  (-1,4) circle (1.5pt);
\draw (-1.22,3.74) node {$1$};
\fill  (1,4) circle (1.5pt);
\draw (1.54,4.24) node {$3$};
\fill  (1.04,2.74) circle (1.5pt);
\draw (1.44,3.18) node {$4$};
\fill  (-0.94,2.68) circle (1.5pt);
\draw (-1.2,2.42) node {$6$};
\fill  (0.02,1.6) circle (1.5pt);
\draw (-0.08,1.26) node {$5$};
\fill  (-2,5) circle (1.5pt);
\fill  (-1.16,5.66) circle (1.5pt);
\fill  (0,6) circle (1.5pt);
\fill  (2.14,2.92) circle (1.5pt);
\fill  (1.34,1.38) circle (1.5pt);
\fill  (-2.28,3.96) circle (1.5pt);
\fill  (-1.52,5.08) circle (1.5pt);
\fill  (0.88,1.88) circle (1.5pt);
\end{scriptsize}
\end{tikzpicture}
\caption{$G_2$}
\end{figure}
\end{minipage}
\begin{minipage}{0.4\linewidth}
Let $W = \{1,2,3\} \sqcup \{4,5\}$. Then it can be seen that
 $G_1=F_2^W(K_6)$ is a $2$-pure fan graph while $G_2=F_2^W(K_6)$ is a $2$-fan
  graph which is not a pure fan graph. 
\end{minipage}
\end{minipage}
\end{example}

In \cite[Lemma 3.2]{dav}, it was proved that $F_k^W(K_n)$ is
Cohen-Macaulay. It may be noted that if $|W_i| > 1$ for some $i$, then
$F_k^W(K_n)$ is neither a closed graph nor a generalized block graph.
In this section, we prove that the regularity of the $k$-fan graph
$F_k^W(K_n)$ is at most the number of maximal cliques in it, i.e., the
class of $k$-fan graphs satisfies the regularity upper bound conjecture
of Saeedi Madani and Kiani. If $k=1$, then we denote $F_1^W(K_n)$ by
$F^W(K_n)$. 
 
\begin{theorem}\label{3.3}
With the above notation, let $G =F_k^W(K_n)$ be a $k$-fan graph of the
complete graph $K_n$ on $W,$ where $n\geq2$.  Then
$\reg(S/J_G) \leq c(G)$. Moreover, if for each $i \in \{1,\ldots,k\} $
and  for each $ j \in \{1,\ldots, r_i\}$, $ a_{i,j} >j+1 $,  then
equality holds. 
\end{theorem}
\begin{proof}
We prove the assertions by induction on $k$. For $k=1$, set $W_1 = \{v_1, \ldots , v_{r_1}\}$.
 We  proceed by induction on $|W_1|= r_1$. If $r_1 =1$, then result
 follows from \cite[Theorem 3.1]{JNR}. Assume that $r_1 >1$ and that the result is true for $r_1 -1$. 
 Set $G_1= K_{a_1} \setminus v_1$ and $G_2 =F^{W_1 \setminus \{v_1\}}(K_n \setminus v_1)$. Since $G_2$ is the fan graph 
of  $K_n \setminus v_1$ on  $W_1 \setminus \{v_1\}$, by induction, $\reg(S/J_{G_2}) \leq c(G_2)$.
 Note that $G=cone(v_1,G_1 \sqcup G_2)$. By \cite[Theorem 3.19]{KM5}, $$\reg(S/J_G) = \reg(S/J_{G_1}) + \reg(S/J_{G_2}) \leq 1+ c(G_2) =c(G).$$
 Now, assume that $k >1$ and result is true for $k-1$. We proceed by induction on $r_k$. If $r_k =1$, then $G=K_{a_{k,1}} * F_{k-1}^{W\setminus \{v_{k,1}\}}(K_n)$.
 By induction on $k$, $\reg(S/J_{F_{k-1}^{W \setminus \{v_{k,1}\}}(K_n)}) \leq c(F_{k-1}^{W \setminus \{v_{k,1}\}}(K_n))$.
 By \cite[Theorem 3.1]{JNR}, $$\reg(S/J_G) =\reg(S/J_{K_{a_{k,1}}})+ \reg(S/J_{F_{k-1}^{W \setminus \{v_{k,1}\}}(K_n)}) \leq 1+ c(F_{k-1}^{W \setminus \{v_{k,1}\}}(K_n))=c(G).$$ 
 Assume  that $r_k >1$ and the result is true for $r_k  -1$. Since, $v=v_{k,1}$ is not a free vertex, by Remark \ref{2.5}, $J_G =Q_1 \cap Q_2$,
  $Q_1 = J_{G_{v}}$,  $Q_2 =(x_v,y_v)+J_{G\setminus v}$ and $Q_1+Q_2 = (x_v,y_v)+J_{G_v \setminus v}$.   
 Let $G' = F_k^{W \setminus \{v\}}(K_n \setminus v)$. Then $G \setminus v = (K_{a_{k,1}} \setminus v) \sqcup G'$.
 By induction on 
  $r_k$, $\reg(S/J_{G'}) \leq c(G')$ and therefore, $$\reg(S/Q_2)=\reg(S/J_{G\setminus v})=\reg(S/J_{K_{a_{k,1}} \setminus v})+\reg(S/J_{G'})
 \leq 1+c(G\setminus v) = c(G).$$  
 Let $H$ be the complete graph on vertex set $N_G[v]$. Note that $G_{v}$ is ($k-1$)-fan graph of $H$  on  $U = W_1 \sqcup \cdots \sqcup W_{k-1}$ and 
 by induction on $k$, $\reg(S/Q_1)=\reg(S/J_{G_v}) \leq c(G_v)< c(G)$. 
 Also, $G_{v} \setminus v$ is ($k-1$)-fan graph of $H\setminus v$  on the set $U = W_1 \sqcup \cdots \sqcup W_{k-1}$. By induction on $k$,
 $\reg(S/{Q_1+Q_2})=\reg(S/J_{G_v \setminus v})  \leq c(G_v \setminus v) < c(G)$. Hence, by the short exact sequence (\ref{2.6}) and Lemma \ref{2.4}, 
 $\reg(S/J_G)  \leq c(G)$.
\end{proof}
Now, we compute the regularity of $k$-pure fan graph.
This result, along with the regularity of $F_m$'s, helps us to compute
the regularity of Cohen-Macaulay bipartite graphs.
 \begin{theorem}\label{3.4}
Let $G=F_k^W(K_n)$ be a $k$-pure fan graph of $K_n$ on  $W,$
where $n\geq2$. Then $\reg(S/J_G) =k+1$.
 \end{theorem}
 \begin{proof}
  We  prove this by induction on $k$. For $k=1$, let $W_1 = \{v_1, \ldots , v_{r_1}\}$ and $\{K_{a_1},\ldots,K_{a_{r_1}}\}$ be the 
  branch of the fan on $W_1$.
 We prove this assertion by induction on $|W_1|= r_1$.  If $r_1 =1$, then result
 follows from \cite[Theorem 3.1]{JNR}. Assume that $r_1 >1$ and the result is true for $r_1 -1$. 
 Write $K_{a_{k,1}} =\{v_1,w\}$.
Set $G' =F^{W_1 \setminus \{ v_1\}}(K_n \setminus v_1)$. Since, $G'$ is $1$-pure fan graph 
of $K_n \setminus v_1$ on  $W_1 \setminus \{v_1\}$, it follows from the induction hypothesis that, $\reg(S/J_{G'}) =2$.
Note that $G=cone(v_1,w \sqcup G')$. Therefore,
by \cite[Theorem 3.19]{KM5}, $\reg(S/J_G) = \reg(S/J_{G'}) =2$. 
 
 Now, assume that $k >1$ and result is true for $k-1$.  If $r_k =1$,  then 
 $G=K_{a_{k,1}} * F_{k-1}^{W \setminus \{v_{k,1}\}}(K_n)$. By induction on $k$, $\reg(S/J_{F_{k-1}^{W \setminus \{v_{k,1}\}}(K_n)}) =k$.
 By  \cite[Theorem 3.1]{JNR}, $$\reg(S/J_G) =\reg(S/J_{K_{a_{k,1}}})+ \reg(S/J_{F_{k-1}^{W\setminus \{v_{k,1}\}}(K_n)})=k+1.$$ 
 Assume  that $r_k >1$ and the result is true for $r_k  -1$. Since, $v=v_{k,1}$ is not a free vertex, by Remark \ref{2.5}, $J_G =Q_1 \cap Q_2$,
 $Q_1= J_{G_{v}}$, $Q_2=(x_v,y_v)+ J_{G\setminus v}$ and $Q_1+Q_2 =  (x_v,y_v)+J_{G_{v} \setminus v}$.
 Note that $ G\setminus v =w_{k,1} \sqcup  G''$, where $G''$ 
 is $k$-pure fan graph of $K_n \setminus v$ on $W\setminus \{v\}$ and  $K_{a_{k,1}}=\{v,w_{k,1}\}$. 
 By induction on $r_k$, $\reg(S/J_{G''}) =k+1$ and therefore, $$\reg(S/Q_2)=\reg(S/J_{G\setminus v})=\reg(S/J_{G''}) =k+1.$$
 Let  $H$ be the complete graph on the vertex set $N_G[v]$. Note that
 $G_v$ is a ($k-1$)-pure fan graph of $H$ 
 on  $W' = W_1 \sqcup \cdots \sqcup W_{k-1}$. By induction on $k$, $\reg(S/Q_1)= \reg(S/J_{G_{v}}) =k$.
Also, by induction on $k$,
 $\reg(S/{Q_1+Q_2})=\reg(S/J_{G_v \setminus v}) 
 =k$.  
 Now, using the short exact sequence (\ref{2.6}) and Lemma \ref{2.4},
 we get $\reg(S/J_G) =k+1$.
 \end{proof}
It was proved in \cite[Theorem 1.1]{MM} that $\reg(S/J_G) \geq l$,
where $l$ is the length of longest induced path.  Note that if $k \geq 2$,
then for $F_k^W(K_n)$ the longest induced path has length $3$. We
conclude this section by 
obtaining an improved lower bound for this class of graphs.
 \begin{corollary}
Let $G=F_k^W(K_n)$ be a $k$-fan graph of the complete graph $K_n$ on the set $W$, where $n\geq2$. Then $\reg(S/J_G) \geq k+1$.
 \end{corollary}

 \begin{proof}
Let $A=[n]
\sqcup \{w_{i,1}: i=1,\ldots,k\}$, where $w_{i,1} \in V(K_{a_{i,1}})
\setminus [n]$. Then $G[A]$ is the induced subgraph of $G$ which is obtained
by adding a whisker each to $k$ vertices of $K_n$. By applying \cite[Corollary
2.2]{MM} and \cite[Theorem 3.1]{JNR}, we get $\reg(S/J_G) \geq \reg(S/J_{G[A]}) = k+1$.
\end{proof}
 
\section{Regularity of Cohen-Macaulay bipartite graphs}

In this section, we compute the regularity of binomial edge ideals of
Cohen-Macaulay bipartite graphs. As a first step, we compute the
regularity of $F_m$, for $m\geq2$, which are the basic building blocks
of a Cohen-Macaulay bipartite graph.
Note that $F_1$ is $K_2$, therefore, $\reg(S/J_{F_1}) =1$.
 
\begin{proposition}\label{4.1}
 For each $m\geq 2$, $\reg(S/J_{F_m}) =3$.
\end{proposition}
\begin{proof}
We prove the assertion by induction on $m$.
Observe that $F_2$ is a path on $4$ vertices, therefore
$\reg(S/J_{F_2}) = 3$.
 
Assume now that $m \geq 3$ and that the result is true for $m-1$. Since $v=2m-1$ is not a free vertex of $F_m$, by Remark \ref{2.5}, 
 $J_{F_m} = Q_1 \cap Q_2$, 
 $Q_1=J_{{{(F_m)}_v}}$, $Q_2= (x_v,y_v)+J_{{F_m} \setminus v}$ and
 $Q_1+Q_2 = (x_v,y_v)+ J_{{{(F_m)}_v} \setminus v}$. 
Note that ${(F_m)}_v = F^{W'}(H)$ is the $1$-pure fan graph of $H$ on the set $W'=\{2,4,\ldots,2m\}$,
where $H$ is a complete graph on vertex set $N_{F_m}[v]$.
 By Theorem \ref{3.4}, $\reg(S/Q_1)=
 \reg(S/J_{F^{W'}(H)}) =2$. Since, ${F_m} \setminus v = F_{m-1} \sqcup \{2m\}$, by induction on $m$, $\reg(S/Q_2)= \reg(S/J_{F_{m-1}}) =3 $. 
 Note that ${(F_m)}_v \setminus v = F^{W'}(H\setminus v)$ is the $1$-pure fan  graph of $H\setminus v$ on $W'$. It follows from  Theorem \ref{3.4} that  
 $\reg(S/{Q_1+Q_2}) = \reg(S/J_{{{(F_m)}_v} \setminus v}) =2$. Thus, by the short exact sequence (\ref{2.6}) and Lemma \ref{2.4}, $\reg(S/J_{F_m})=3$. Hence, the assertion follows. 
\end{proof}

It may be noted that for $F_m$, any maximal induced path has length $3$.
Therefore, one can say that $F_m$'s have minimal regularity, in the
sense that it attains the lower bound given by Matsuda and Murai,
\cite{MM}.

 \begin{remark}\label{4.2}
For the operation $\circ$, in \cite{dav}, the authors
assumed that $\deg_{G_i}(v_i) \geq 3$, for each $i$. By allowing
$\deg_{G_i}(v_i) = 2$, we can apply the operation $\circ$ with $F_2$ as one
of the graphs. 
If $F_{m_1}$ is a graph with $m_1\geq 2$, then $F_{m_1} \circ F_2 =F_{m_1}*F_1$. By \cite[Theorem 3.1]{JNR}, 
$\reg(S/J_{F_{m_1} \circ F_2})
 =\reg(S/J_{F_{m_1} * F_1}) = 3+1 =4$.
 \end{remark}
We now compute the regularity of 
$F_{m_1} \circ F_{m_2}$
in terms of the regularities of $F_{m_1}$ and $F_{m_2}$.
\begin{proposition}\label{4.3}
Let $m_1,m_2\geq 3$ and $G=(F_{m_1},f_1) \circ (F_{m_2},f_2)$. Then 
\[\reg(S/J_G)=\reg(S/J_{F_{m_1 -1}})+\reg(S/J_{F_{m_2 -1}}) =6.\]
 \end{proposition}
\begin{proof}
Let $V(F_{m_1}) = \{u_1, \ldots, u_{2m_1}\}$ and $V(F_{m_2}) =
\{w_1, \ldots, w_{2m_2}\}$. In $F_{m_1}$, there are two vertices of
degree $1$, namely $u_1$ and $u_{2m_1}$. So is the case for $F_{m_2}$.
It may be noted that the graphs obtained by different choices of $f_1$ and
$f_2$ are isomorphic. Hence, without loss of generality, we may assume
that $f_1 = u_{2m_1}$ and $f_2 = w_{2m_2}$. Let $v = u_{2m_1-1} =
w_{2m_2-1}$ in $G$.
Since, $v$ is not a free vertex of the graph $G$, by Remark \ref{2.5},
there exist $Q_1 = J_{G_v} $ and $Q_2 =(x_v,y_v) +J_{G\setminus v}$ so
that $J_G = Q_1 \cap Q_2$ and $Q_1+Q_2 = (x_v,y_v)+J_{G_v \setminus v}$.
Let $H$ be the complete graph on vertex set $N_G[v].$
Note that $G_v = F_2^W(H)$ is a $2$-pure fan graph of $H$ on $W =
N_G(v)$, $G \setminus v = F_{m_1-1} \sqcup F_{m_2-1}$ and $G_v
\setminus v = F_2^W(H \setminus v)$ which is a $2$-pure fan of $H
\setminus v$ on $W$. Therefore, it follows from Theorem \ref{3.4} and
Proposition \ref{4.1} that
\[
  \reg(S/Q_1)=3, \reg(S/Q_2) = \reg(S/J_{F_{m_1 -1}})+\reg(S/J_{F_{m_2 -1}})=6 \text{ and }
 \reg(S/{Q_1+Q_2}) =3.\]
Hence, it follows from the short exact sequence (\ref{2.6}) and Lemma
\ref{2.4} that $$\reg(S/J_G) =\reg(S/J_{F_{m_1 -1}})+\reg(S/J_{F_{m_2 -1}})= 6.$$
\end{proof}
For the rest  of the section, we assume that $F_k^W(K_n)$ is a
$k$-pure fan graph.  
\begin{proposition}\label{4.4}
For $m,n\geq3$, let $G=(F_m,f_1) \circ (F_k^W(K_n),f_2)$, where $W =
W_1 \sqcup \cdots \sqcup W_k \subseteq [n]$. Write $v = v_1 = v_2$ in $G$.
Assume that $|W_i|\geq 2$ for some $i$ and $v \in W_i$. Then 
\[\reg(S/J_G)= \reg(S/J_{F_{m -1}})+\reg(S/J_{F_k^W(K_n) \setminus
\{v,f_2\}})=k+4.\]
 \end{proposition}
\begin{proof}
Without loss of generality, assume that $|W_1| \geq 2$ and $v \in
W_1$.
Since $v$ is not a free vertex of $G$, by Remark \ref{2.5}, there exist
$Q_1 = J_{G_v}$ and $Q_2 =(x_v,y_v)+J_{G\setminus v}$ such that 
$J_G=Q_1 \cap Q_2$ and $Q_1+Q_2 = (x_v,y_v)+ J_{G_v \setminus v}$. Let
$H$ be the  complete graph on $N_G[v]$. Note that $G_v =F_k^{W'}(H)$
is a $k$-pure fan  of $H$ on $W'=N_{F_m \setminus f_1}(v) \sqcup
(W \setminus W_1)$.
By Theorem \ref{3.4}, $\reg(S/Q_1)= \reg(S/J_{G_v}) = k+1$. Since $K_{a_{1,1}} =\{v,f_2\}$ and
$G\setminus v = F_{m-1} \sqcup (F_k^{W}(K_n)\setminus \{v,f_2\})$,  by
Proposition \ref{4.1} and Theorem \ref{3.4}, we get

$$\reg(S/Q_2) =\reg(S/J_{G\setminus v}) = \reg(S/J_{F_{m-1}})+ \reg(S/J_{F_k^W(K_n) \setminus \{v,f_2\}})  =k+4.$$
Since $G_v \setminus v$ is an induced subgraph of $G_v$,
$\reg(S/(Q_1+Q_2)) = \reg(S/J_{G_v \setminus v}) \leq \reg(S/J_{G_v}) = k+1$.
Hence we  conclude from the short exact sequence (\ref{2.6}) and
Lemma \ref{2.4} that 
$$\reg(S/J_G) =\reg(S/J_{F_{m -1}})+\reg(S/J_{F_k^W(K_n)\setminus \{v,f_2\}})= k+4.$$
 \end{proof}

\begin{remark}\label{4.5}
\begin{enumerate}
   \item In Proposition \ref{4.4}, we had assumed that $m\geq 3$. If $m=2$,
	then $F_m$ is a path of length $3$. 
  Since $G=F_2 \circ F_k^W(K_n) = F_1 * F_k^W(K_n)$, by \cite[Theorem 3.1]{JNR},
  $$\reg(S/J_G) =\reg(S/J_{F_1})+\reg(S/J_{F_k^W(K_n)})=k+2.$$
  
  \item We had also assumed that $|W_i| \geq 2$ for some $i$.
	If  $|W_i|=1$ for each $i$, then observe that $G \setminus v = F_{m-1} \sqcup
	F_{k-1}^{W \setminus \{v\}}(K_n \setminus v)$ and hence
	$\reg(S/Q_2) = k+3$. Note that the regularities of $S/Q_1$ and
	$S/(Q_1+Q_2)$ remain the same. Therefore, it follows that
	$\reg(S/J_G) = k+3$.
 \end{enumerate}
 \end{remark}

We now study the regularity of graphs obtained by composing several
$F_m$'s with a pure fan graph using the operation $\circ$.

\begin{theorem}\label{4.6}
Let $n \geq 3$ and $H$ denote either $F_n$ or $F_k^W(K_n)$ with $W = W_1 \sqcup
\cdots \sqcup W_k$ and $|W_i| \geq 2$ for some $i$. Let $G= F_{m_1} \circ \cdots \circ F_{m_t} \circ (H,f)$ be a graph with $t\geq2$ 
and for each $i\in [t]$, $m_i \geq 3$. Let $V(F_{m_1} \circ \cdots
\circ F_{m_t}) \cap V(H)= \{v\}$ and $f$ be a pendant vertex in $N_H(v)$. 
If $H = F_k^W(K_n)$, then assume that $v\in W_i$ and $|W_i| \geq 2$.
Then  
$$\reg(S/J_G) = \reg(S/F_{m_1 -1})+ \reg(S/F_{m_2 -2})+\cdots + \reg(S/F_{m_t -2})+\reg(S/J_{H\setminus \{v,f\}}).$$
\end{theorem}
\begin{proof}
For each $i \in \{1,\ldots,t\}$ and $j=\{1,2\}$,
let $f_{i,j}$ be the only pendant vertices of $F_{m_i}$ and for each 
$i \in\{1,\ldots,t-1\}$, $V(F_{m_i}) \cap V(F_{m_{i+1}})=\{v_{i,i+1}\}$, i.e. $F_{m_i} \circ F_{m_{i+1}}$
is the graph obtained from $F_{m_i}$ and $F_{m_{i+1}}$ by removing the pendant vertices $f_{i,2},f_{i+1,1}$ 
and identifying the vertices $2m_i-1 =v_{i,i+1} =2$.
Following  Remark \ref{2.5}, set 
$J_G = Q_1 \cap Q_2$, $ Q_1 = J_{G_v}$, $Q_2 =(x_v,y_v)+J_{G\setminus v} $
and $Q_1+Q_2 =(x_v,y_v)+ J_{G_v \setminus v}$.

We  proceed by  induction  on $t\geq 2$. Let $t=2$.  Let
$H=F_k^W(K_n)$ and $H'$ be the complete graph on $N_G[v]$. Without
loss of generality, assume that $|W_1| \geq 2$ and $v \in W_1$. Note that
$G_v = F_{m_1} \circ G'$, where $G'=F_k^{W'}(H')$ is the $k$-pure fan
graph of $H'$ on $W'=N_{F_{m_2} \setminus f_{2,2}}(v) \sqcup (W 
\setminus W_1)$. Since $m_2 \geq 3$, it follows from Proposition \ref{4.4} that $\reg(S/Q_1) =\reg(S/J_{G_v})
=k+4$.  Note that $G\setminus v = F_{m_1} \circ F_{m_2 -1} \sqcup
H\setminus \{v,f\}$. By Proposition \ref{4.3} and Remark \ref{4.2},
$$\reg(S/Q_2) =\reg(S/J_{G\setminus v}) =
\reg(S/J_{F_{m_1-1}})+\reg(S/J_{F_{m_2-2}})+ k+1.$$ Since $G_v
\setminus v$ is an induced subgraph of $G_v$, $\reg(S/J_{F_{m_1-1}})
= 3$ and $\reg(S/J_{F_{m_2-2}}) \geq 1$, it follows from the short
exact sequence (\ref{2.6}) and Lemma \ref{2.4} that
$$\reg(S/J_G)=\reg(S/J_{F_{m_1-1}})+\reg(S/J_{F_{m_2-2}})+ k+1.$$

Assume now that $H=F_n$. Without loss of generality, assume that $v =
2n-1$. Let $H''$ be the complete graph on $N_G[v]$
and $G''=F_2^{W''}(H'')$ is $2$-pure fan  of
$H''$ on $W''=N_{F_{m_2} \setminus f_{2,2}}(v) \sqcup N_{F_{n} \setminus f}(v)$.
Then $G_v = F_{m_1} \circ G''$.
By Proposition \ref{4.4}, $\reg(S/Q_1) =\reg(S/J_{G_v}) =6$.
Since $G\setminus v = F_{m_1} \circ F_{m_2 -1} \sqcup F_{n-1}$, by
Proposition \ref{4.3} and Remark \ref{4.2},  
$$\reg(S/Q_2) =\reg(S/J_{G\setminus v}) =
\reg(S/J_{F_{m_1-1}})+\reg(S/J_{F_{m_2-2}})+ \reg(S/J_{F_{n-1}}) \geq
7.$$
 Note that $G_v\setminus v$ is an induced subgraph of $G_v$. Thus, by \cite[Corollary 2.2]{MM}, $\reg(S/Q_1+Q_2) = \reg(S/J_{G_v \setminus v}) \leq \reg(S/J_{G_v}) = 6$.
Hence, it follows from the short exact sequence (\ref{2.6}) and 
Lemma \ref{2.4} that  
$$\reg(S/J_G)=\reg(S/J_{F_{m_1-1}})+\reg(S/J_{F_{m_2-2}})+ \reg(S/J_{F_{n-1}}).$$
Now, assume that $t\geq3$ and the result  is true for $\leq t-1$. 
Let $H=F_k^W(K_n)$ and $H_1$ be the complete graph on $N_G[v]$.
Note that $G_v = F_{m_1} \circ \cdots \circ F_{m_{t-1}} \circ G_1$, 
where $G_1=F_k^{U}(H_1)$ is the $k$-pure fan  of 
$H_1$ on 
$U=N_{F_{m_t} \setminus f_{t,2}}(v) \sqcup (W \setminus 
W_1)$, $G\setminus v = F_{m_1} \circ \cdots \circ F_{m_{t-1}} \circ
F_{m_t-1} \sqcup H\setminus \{v,f\}$ and $G_v \setminus v = F_{m_1} \circ \cdots \circ F_{m_{t-1}}
\circ F_k^U(H_1\setminus v)$. Hence by induction on $t$,
\begin{eqnarray*}
  \reg(S/J_{G_v}) & =& \reg(S/F_{m_1 -1})+ \reg(S/F_{m_2
  -2})+\cdots + \reg(S/F_{m_{t-1} -2})+k+1; \\
  \reg(S/J_{G\setminus v}) & =&  
  \reg(S/J_{F_{m_1-1}})+ \reg(S/J_{F_{m_2 -2}}) +\cdots +
  \reg(S/J_{F_{m_{t-1}-2}})\\
  & & +\reg(S/J_{F_{m_t-2}})+ k+1; \\
  \reg(S/J_{G_v \setminus v}) & =&  \reg(S/F_{m_1 -1})+ \reg(S/F_{m_2 -2})+\cdots + \reg(S/F_{m_{t-1} -2})+k+1.
\end{eqnarray*}
By (\ref{2.6}) and  Lemma \ref{2.4}, we get 
  $$\reg(S/J_G)=\reg(S/F_{m_1 -1})+ \reg(S/F_{m_2 -2})+\cdots + \reg(S/F_{m_{t} -2})+k+1.$$
  
Now assume that $H=F_n$. Let $H_2$ be the complete graph on vertex set $N_G[v]$.
Note that $G_v = F_{m_1} \circ \cdots \circ F_{m_{t-1}} \circ G_2$, 
where $G_2=F_2^{U'}(H_2)$ is the $2$-pure  fan  of 
$H_2$ on $U'=N_{F_{m_t} \setminus f_{t,2}}(v) \sqcup N_{F_{n}
\setminus f}(v) $, $G\setminus v = F_{m_1} \circ \cdots \circ
F_{m_{t-1}} \circ F_{m_t -1} \sqcup F_{n-1}$ and $G_v \setminus v =
F_{m_1} \circ \cdots \circ F_{m_{t-1}} \circ F_2^{U'}(H_2\setminus
v)$.  Hence by induction on $t$, 
\begin{eqnarray*}
\reg(S/J_{G_v}) & = & \reg(S/F_{m_1 -1})+ \reg(S/F_{m_2 -2})+\cdots +
\reg(S/F_{m_{t-1} -2})+3;\\
\reg(S/J_{G\setminus v}) & = & 
 \reg(S/J_{F_{m_1-1}})+\reg(S/J_{F_{m_2-2}})+\cdots+ \reg(S/J_{F_{m_t
 -2}}) + \reg(S/J_{F_{n-1}}); \\
\reg(S/J_{G_v \setminus v}) & = & \reg(S/F_{m_1 -1})+\reg(S/F_{m_2 -2})+\cdots + \reg(S/F_{m_{t-1} -2})+3.
\end{eqnarray*}
Using the short exact sequence (\ref{2.6}) and Lemma \ref{2.4}, we
conclude that
  $$\reg(S/J_G)=  \reg(S/J_{F_{m_1-1}})+\reg(S/J_{F_{m_2-2}})+\cdots+ \reg(S/J_{F_{m_t -2}}) + \reg(S/J_{F_{n-1}}).$$ Hence, the assertion follows.
  \end{proof}
Now, we obtain a precise expression for regularity of binomial edge ideal of Cohen-Macaulay bipartite graphs.
By \cite[Theorem 6.1]{dav}, if $G$ is a connected Cohen-Macaulay
bipartite graph, then there exists a positive integer $s$ such that
$G=G_1*\cdots*G_s$, where $G_i=F_{n_i}$ or $G_i=F_{m_{i,1}} \circ \cdots \circ F_{m_{i,t_i}}$, 
for some $n_i\geq1$ and $m_{i,j}\geq 3$ for each $j=1,\ldots,t_i$. Let
$A=\{i\in[s]: G_i =F_{n_i} , n_i \geq2\}$, $B=\{i\in[s]: G_i =F_{n_i}
, n_i =1\}$ and 
$C=\{i\in [s]: G_i=F_{m_{i,1}} \circ \cdots \circ F_{m_{i,t_i}},
t_i\geq 2\}$. For each $i\in C$, let $C_i=\{j \in \{2, \ldots, t_i-1\}
~ : ~m_{i,j} \geq 4  \}\sqcup \{1,t_i\}$ and 
$C_i'=\{j \in \{2, \ldots, t_i-1\} ~ : ~
m_{i,j} = 3 \}$. Set $\alpha = |A| + \sum_{i \in C} |C_i|$ and $\beta
= |B| + \sum_{i \in C}|C_i'|$.
\begin{theorem}\label{cm-bipartite}
Let $G=G_1*\cdots*G_s$ be Cohen-Macaulay connected bipartite graph.
Let $\alpha$ and $\beta$ be as defined above. Then
  $\reg(S/J_G) =3\alpha+\beta$.
 \end{theorem}

\begin{proof}
By \cite[Theorem 3.1]{JNR}, $$\reg(S/J_G) =  \sum_{i=1}^s
\reg(S/J_{G_i}). $$ By Proposition \ref{4.1}, $\reg(S/J_{G_i}) = 3$ for $i
\in A$. If $i
\in B$, then $\reg(S/J_{G_i}) = 1$. If $i \in C$, then it follows from
Theorem \ref{4.6} that $\reg(S/J_{G_i}) = 3|C_i| + |C_i'|$. Hence the
assertion follows.
\end{proof}

We illustrate our result in the following example. 
Let $G= F_3 \circ F_4 \circ F_3 \circ F_3 \circ F_3 $ be the graph
as shown in figure below

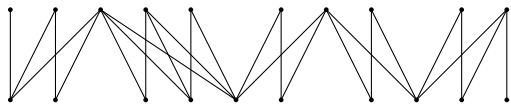
\begin{figure}[H]
\begin{tikzpicture}[scale=.6]
\draw (-4,0)-- (-4,2);
\draw (-4,0)-- (-3,2);
\draw (-4,0)-- (-2,2);
\draw (-3,2)-- (-3,0);
\draw (-3,0)-- (-2,2);
\draw (-2,2)-- (-1,0);
\draw (-2,2)-- (0,0);
\draw (-1,0)-- (-1,2);
\draw (0,0)-- (0,2);
\draw (-1,2)-- (0,0);
\draw (-1,2)-- (1,0);
\draw (0,2)-- (1,0);
\draw (-2,2)-- (1,0);
\draw (1,0)-- (2,2);
\draw (1,0)-- (3,2);
\draw (2,2)-- (2,0);
\draw (2,0)-- (3,2);
\draw (3,2)-- (4,0);
\draw (3,2)-- (5,0);
\draw (4,2)-- (4,0);
\draw (4,2)-- (5,0);
\draw (5,0)-- (6,2);
\draw (5,0)-- (7,2);
\draw (6,2)-- (6,0);
\draw (7,2)-- (7,0);
\draw (7,2)-- (6,0);
\begin{scriptsize}
\fill  (-3,2) circle (1.5pt);
\fill  (-1,2) circle (1.5pt);
\fill  (-3,0) circle (1.5pt);
\fill  (-1,0) circle (1.5pt);
\fill  (-2,2) circle (1.5pt);
\fill  (0,0) circle (1.5pt);
\fill  (0,2) circle (1.5pt);
\fill  (1,0) circle (1.5pt);
\fill  (2,2) circle (1.5pt);
\fill  (2,0) circle (1.5pt);
\fill  (3,2) circle (1.5pt);
\fill  (4,2) circle (1.5pt);
\fill  (4,0) circle (1.5pt);
\fill  (5,0) circle (1.5pt);
\fill  (6,2) circle (1.5pt);
\fill  (6,0) circle (1.5pt);
\fill  (7,2) circle (1.5pt);
\fill  (7,0) circle (1.5pt);
\fill  (-4,2) circle (1.5pt);
\fill  (-4,0) circle (1.5pt);
\end{scriptsize}
\end{tikzpicture}
\end{figure}

Note that $G$ is Cohen-Macaulay bipartite graph. With respect to the
notation in Theorem \ref{cm-bipartite}, $A = \emptyset = B$ and $C =
\{1\}$. Also, we have $|C_1| = 3$ and $|C_1'| = 2$. Therefore,
by Theorem \ref{cm-bipartite}
$\reg(S/J_G) =11$.

 \vskip 2mm
\noindent
\textbf{Acknowledgements:} The second author thanks the National Board
for Higher Mathematics, India for the financial support. We have extensively
used SAGE \cite{sage} and Macaulay 2 \cite{M2} for computational
purposes.
 
 %\nocite*{}
\bibliographystyle{plain}  %% or 
\bibliography{bipartite}

\begin{thebibliography}{10}

\bibitem{dav}
Davide Bolognini, Antonio Macchia, and Francesco Strazzanti.
\newblock Binomial edge ideals of bipartite graphs.
\newblock {\em European J. Combin.}, 70:1--25, 2018.

\bibitem{sage}
The~Sage Developers.
\newblock {\em {S}age {M}athematics {S}oftware ({V}ersion 6.9)}, 2015.
\newblock {\tt http://www.sagemath.org}.

\bibitem{her1}
Viviana Ene, J\"urgen Herzog, and Takayuki Hibi.
\newblock Cohen-{M}acaulay binomial edge ideals.
\newblock {\em Nagoya Math. J.}, 204:57--68, 2011.

\bibitem{M2}
Daniel~R. Grayson and Michael~E. Stillman.
\newblock Macaulay2, a software system for research in algebraic geometry.
\newblock Available at \url{http://www.math.uiuc.edu/Macaulay2/}.

\bibitem{HH1}
J\"urgen Herzog, Takayuki Hibi, Freyja Hreinsd\'ottir, Thomas Kahle, and
  Johannes Rauh.
\newblock Binomial edge ideals and conditional independence statements.
\newblock {\em Adv. in Appl. Math.}, 45(3):317--333, 2010.

\bibitem{JNR}
A.~V. {Jayanthan}, N.~{Narayanan}, and B.~V. {Raghavendra Rao}.
\newblock {Regularity of Binomial Edge Ideals of Certain Block Graphs}.
\newblock {\em Proceedings - Mathematical Sciences}, To Appear.

\bibitem{KM5}
D.~{Kiani} and S.~{Saeedi Madani}.
\newblock {The regularity of binomial edge ideals of graphs}.
\newblock {\em ArXiv e-prints}, October 2013.

\bibitem{KM3}
Dariush Kiani and Sara Saeedi~Madani.
\newblock The {C}astelnuovo-{M}umford regularity of binomial edge ideals.
\newblock {\em J. Combin. Theory Ser. A}, 139:80--86, 2016.

\bibitem{MM}
Kazunori Matsuda and Satoshi Murai.
\newblock Regularity bounds for binomial edge ideals.
\newblock {\em J. Commut. Algebra}, 5(1):141--149, 2013.

\bibitem{oh}
Masahiro Ohtani.
\newblock Graphs and ideals generated by some 2-minors.
\newblock {\em Comm. Algebra}, 39(3):905--917, 2011.

\bibitem{Rauf}
Asia Rauf and Giancarlo Rinaldo.
\newblock Construction of {C}ohen-{M}acaulay binomial edge ideals.
\newblock {\em Comm. Algebra}, 42(1):238--252, 2014.

\bibitem{KM1}
Sara Saeedi~Madani and Dariush Kiani.
\newblock Binomial edge ideals of graphs.
\newblock {\em Electron. J. Combin.}, 19(2):Paper 44, 6, 2012.

\bibitem{KM2}
Sara Saeedi~Madani and Dariush Kiani.
\newblock On the binomial edge ideal of a pair of graphs.
\newblock {\em Electron. J. Combin.}, 20(1):Paper 48, 13, 2013.

\end{thebibliography}

\end{document}